\documentclass[10pt]{article}
\usepackage[margin=1in]{geometry}
\usepackage{amsfonts,amssymb,amsmath,amsthm}
\usepackage{url}
\usepackage{enumerate}
\usepackage{amsfonts,epic,amsbsy,amscd,amstext}
\usepackage{amssymb,latexsym,amsmath,amsthm,amscd,mathrsfs}
\usepackage{comment}
\usepackage{color}
\usepackage{mathrsfs}
\usepackage{bbold}
\newtheorem{theorem}{Theorem}[section]

\newtheorem{question}[theorem]{Question}
\newtheorem{lemma}[theorem]{Lemma}
\newtheorem{proposition}[theorem]{Proposition}
\newtheorem{corollary}[theorem]{Corollary}

\theoremstyle{definition}

\newtheorem{definition}[theorem]{Definition}
\newtheorem{observation}[theorem]{Observation}
\numberwithin{equation}{section}

\usepackage{etoolbox}

\makeatletter
\patchcmd{\@settitle}{\uppercasenonmath\@title}{}{}{}
\patchcmd{\@setauthors}{\MakeUppercase}{}{}{}
\patchcmd{\section}{\scshape}{}{}{}
\makeatother

\newcommand{\zfc}{\mathnormal{\mathsf{ZFC}}}

\newcommand{\ch}{\mathnormal{\mathsf{CH}}}

\newcommand{\covm}{\mathnormal{\mathrm{cov}(\mathcal M)}}

\DeclareMathOperator{\cf}{cf}
\DeclareMathOperator{\id}{id}
\DeclareMathOperator{\ran}{ran}
\DeclareMathOperator{\ps}{Ps}
\DeclareMathOperator{\dom}{dom}
\DeclareMathOperator{\shr}{Shrink}
\DeclareMathOperator{\cd}{Code}
\DeclareMathOperator{\dcd}{Dcd}
\DeclareMathOperator{\inpa}{IntPar}
\DeclareMathOperator{\coh}{Coh}
\begin{document}

\author{David~J. Fern\'andez-Bret\'on}

\title{Generalized pathways}

\date{}
%\date{March 14, 2017--\today}

\maketitle

\begin{abstract}
We study a type of object, called a pathway (generalizing pathways in the sense of P. E. Cohen~\cite{cohen-random}), which is useful for several set-theoretic constructions and whose existence, in a sense, generalizes the notion of a cardinal characteristic being large.
\end{abstract}

\section{Introduction}

In the paper~\cite{cohen-random}, P.E. Cohen introduced a device, which he called a \textit{pathway}, to prove the existence of P-points in models obtained by assuming $\ch$ in the ground model, and afterwards forcing with the measure algebra. Thirty seven years later, Michael Hru\v{s}\'ak and the author~\cite{gruffpaper} used what they called a \textit{strong pathway}, similar to one of Cohen's pathways but with stronger properties, to show that there are gruff ultrafilters in the same kind of Random model. Unfortunately, both proofs were wrong, as pointed out by Osvaldo Guzm\'an (see~\cite{corrigendum}). In both cases, the error lies in the proof that the corresponding variety of pathway exists in the Random model, and the part of the proof where the existence of the relevant ultrafilter is derived from the existence of the corresponding pathway is correct. Furthermore, the paper~\cite{gruffpaper} also contains a proof that the existence of a gruff ultrafilter follows from $\mathfrak d=\mathfrak c$. In this note we investigate what the three proofs have in common, by introducing a generalized definition of a \textit{pathway}, which subsumes both varieties of pathway mentioned above, as well as the assumption that $\mathfrak d=\mathfrak c$. The definitions are stated in a great deal of generality, which allows us to consider a wide variety of different pathways, and we show how some of these can be used to recover classical results on consistent existence of set-theoretic objects.

\section{The generalized notion of a pathway}

For this note, we will restrict ourselves to contemplating cardinal characteristics of the continuum that are evaluations of triples. That is, a cardinal characteristic will be computed from a triple $\langle A,B,E\rangle$, where $|A|,|B|\leq\mathfrak c$ and $E\subseteq A\times B$ by looking at the cardinal\footnote{Technically, in order for our cardinal to be well defined we need to also require that $(\forall a\in A)(\exists b\in B)(a\ E\ b)$. It is also customary to require that $(\forall b\in B)(\exists a\in A)\neg(a\ E\ b)$, to ensure that the resulting cardinal is at least $2$; in most cases of interest, we actually have that $(\forall X\in[B]^{\leq\omega})(\exists a\in A)(\forall b\in X)\neg(a\ E\ b)$, so that the corresponding cardinal characteristic is uncountable.}
\begin{equation*}
\mathfrak x=\min\{|X|\big|X\subseteq B\wedge(\forall a\in A)(\exists b\in X)(a\ E\ b)\}.
\end{equation*}
Usually we will identify the cardinal characteristic with the triple used to compute it, although in some cases we will have to emphasize the specific triple that we are thinking of, since it might happen that different triples with different properties give rise to the same cardinal characteristic.

\begin{definition}\label{defpathway}
Let $\mathfrak x=\langle A,B,E\rangle$ be a cardinal characteristic of the continuum, and let $\mathscr F$ be a collection of finitary functions on $B$ (that is, each $F\in\mathscr F$ is $F:B^n\longrightarrow B$ for some $n<\omega$). Also let $\kappa$ be a regular cardinal. An \textbf{$\langle\mathfrak x,\mathscr F,\kappa\rangle$-pathway} is a continuous increasing sequence $\langle B_\alpha\big|\alpha<\kappa\rangle$ of subsets of $B$ such that:
\begin{enumerate}
\item $\bigcup_{\alpha<\kappa}B_\alpha=B$,
\item $(\forall\alpha<\kappa)(\exists x\in A)(\forall y\in B_\alpha)\neg(x\ E\ y)$, and 
\item each $B_\alpha$ is closed under $\mathscr F$, that is, for each $F\in\mathscr F$ of arity $n$, we have that $F[B_\alpha^n]\subseteq B_\alpha$.
\end{enumerate}
\end{definition}

\begin{observation}
The first thing to notice is that, in principle, it would make sense to define pathways of any length (not just of regular cardinal length). However, the conjunction of requirements 1 and 2 in Definition~\ref{defpathway} forces such length to be a limit ordinal. Moreover, given $\mathfrak x$, $\mathscr F$, and $\alpha$, if there is an $\langle\mathfrak x,\mathscr F,\alpha\rangle$-pathway then there is also an $\langle\mathfrak x,\mathscr F,\cf(\alpha)\rangle$-pathway (by thinning out the original pathway). Thus we will restrict ourselves to studying pathways whose length is a regular cardinal, the way we stated it in Definition~\ref{defpathway}.
\end{observation}

\begin{observation}
Note that, if $\langle B_\alpha\big|\alpha<\kappa\rangle$ is an $\langle\mathfrak x,\mathscr F,\kappa\rangle$-pathway, then there is a collection $\{x_\alpha\big|\alpha<\kappa\}\subseteq A$ which is $\neg E^{-1}$-dominating (just let $x_\alpha\in A$ be such that $(\forall y\in B_\alpha)\neg(x_\alpha\ E\ y)$, for each $\alpha<\kappa$) and thus it must be the case that $\mathfrak x^*\leq\kappa$ (here $\mathfrak x^*=\langle B,A,\neg E^{-1}\rangle$ is the cardinal characteristic dual to $\mathfrak x$). Thus pathways cannot be ``too short''; they cannot be ``too long'' either, for if we demand that the sequence $\langle B_\alpha\big|\alpha<\kappa\rangle$ be strictly increasing, then clearly its length must be strictly less than $|B|^+$ (and consequently, less than or equal to $\cf(|B|)$, given that we are only considering regular cardinals as possible lengths of pathways).
\end{observation}

\begin{observation}\label{strongercondition}
In the case of cardinal characteristics $\mathfrak x=\langle A,B,E\rangle$ where $A=B$, it is possible, given a $\langle\mathfrak x,\mathscr F,\kappa\rangle$-pathway, to thin out the sequence to ensure that condition (2) in Definition~\ref{defpathway} is replaced by
\begin{enumerate}
\item[(2')]$(\forall\alpha<\kappa)(\exists x\in B_{\alpha+1})(\forall y\in B_\alpha)\neg(x\ E\ y)$.
\end{enumerate}
\end{observation}

\section{Pathways and cardinal characteristics}

The main idea behind the existence of an $\langle\mathfrak x,\mathscr F,\kappa\rangle$-pathway is that this is, in a sense, a generalization of the notion of ``$\mathfrak x$ is as large as possible'', in the sense that the existence of such a pathway will allow us, in most cases (as we will show here later on), to prove many of the things that are consequences of $\mathfrak x$ being as large as possible; and this is done by means of a recursive construction where, instead of having a set which is ``too small'' to be $E$-dominating, at each stage one uses $B_\alpha$ to achieve the same effect. To make this intuitive notion precise, we proceed to prove the result in this section. We start with an easy counting argument.

\begin{lemma}\label{closureofone}
Let $F:Z^n\longrightarrow Z$ be a finitary operation on an infinite set $Z$, and let $Y\subseteq Z$ with $|Y|=\lambda$. Then there exists a set $Y(F)\subseteq Z$ which is closed under $F$ and such that $|Y(F)|=\max\{\lambda,\omega\}$.
\end{lemma}

\begin{proof}
Recursively define sets $Y=Y_0\subseteq Y_1\subseteq\cdots\subseteq Y_k\subseteq\cdots\subseteq Z$ as follows: given $Y_k$, let $Y_{k+1}=Y_k\cup\{F(\vec{x})\big|\vec{x}\in Y_k^n\}$. Clearly $|Y_{k+1}|\leq|Y_k|+|Y_k|^n\leq\max\{\omega,|Y_k|\}$. Then we just let $Y(F)=\bigcup_{k<\omega}Y_k$, which clearly satisfies what is claimed.
\end{proof}

\begin{lemma}\label{closureofmany}
Let $\mathscr F$ be a collection of finitary operations on some infinite set $Z$, and let $Y\subseteq Z$ with $|Y|=\lambda$. If $|\mathscr F|=\kappa$, then there exists a set $Y(\mathscr F)$ closed under every $F\in\mathscr F$ with $|Y(\mathscr F)|=\max\{\lambda,\kappa,\omega\}$.
\end{lemma}

\begin{proof}
Let $\langle F_\alpha\big|\alpha<\max\{\kappa,\omega\}\rangle$ be an enumeration of the elements of $\mathscr F$ such that each $F\in\mathscr F$ occurs cofinally often in the enumeration. Recursively construct a continuous increasing sequence $\langle Y_\alpha\big|\alpha<\max\{\kappa,\omega\}\rangle$ such that $Y_{\alpha+1}$ is closed under $F_\alpha$, and each $|Y_\alpha|\leq\max\{|\alpha|,\lambda,\omega\}$, by letting $Y_{\alpha+1}=Y_\alpha(F_\alpha)$ as in Lemma~\ref{closureofone} (and in the limit stages just take unions). In the end, it suffices to let $Y(\mathscr F)=\bigcup_{\alpha<\max\{\kappa,\omega\}}Y_\alpha$.
\end{proof}

Now for the main result of this section.

\begin{theorem}\label{characteristiclarge}
Let $\mathfrak x=\langle A,B,E\rangle$ be a cardinal characteristic of the continuum. Suppose that $\kappa:=\mathfrak x=|B|$ (i.e. $\mathfrak x$ is as large as possible), and let $\mathscr F$ be any family of finitary operations on $B$ with $|\mathscr F|<\kappa$. Then there is an $\langle\mathfrak x,\mathscr F,\kappa\rangle$-pathway.
\end{theorem}

\begin{proof}
Let $\langle y_\alpha\big|\alpha<\kappa\rangle$ be an enumeration of the elements of $B$. Recursively construct a continuous increasing sequence $\langle B_\alpha\big|\alpha<\kappa\rangle$ as follows. $B_0=\{y_0\}(\mathscr F)$, which has cardinality at most $\max\{|\mathscr F|,\omega\}$. At the limit stage $\alpha$, we just take $B_\alpha=\bigcup_{\xi<\alpha}B_\xi$, and notice that the cardinality is still at most $\max\{|\alpha|,|\mathscr F|\}<\kappa$. And given $B_\alpha$ with $|B_\alpha|\leq\{|\alpha|,|\mathscr F|\}$, we just let $B_{\alpha+1}=(B_\alpha\cup\{y_\alpha\})(\mathscr F)$, which by the previous lemma should still have cardinality at most $\max\{|\alpha|,|\mathscr F|\}$. Essentially by definition, $\bigcup_{\alpha<\kappa}B_\alpha=B$ and each $B_\alpha$ is closed under every operation $F\in\mathscr F$. To conclude that the sequence $\langle B_\alpha\big|\alpha<\kappa\rangle$ thus constructed in fact constitutes a $\langle\mathfrak x,\mathscr F,\kappa\rangle$-pathway, all that is left is to notice that, since $|B_\alpha|<\kappa=\mathfrak x$, $B_\alpha$ cannot be $E$-dominating, that is, there exists a $x\in A$ such that for each $y\in B_\alpha$, $\neg(x\ E\ y)$.
\end{proof}

\section{Applications of pathways for $\mathfrak d$}

For us, the cardinal characteristic $\mathfrak d$ will be the evaluation of the triple $\langle\omega^\omega,\omega^\omega,\leq^*\rangle$. Pathways involving this cardinal characteristic are of special importance, since they were the first ones ever considered. As explained in the Introduction, Paul E. Cohen~\cite{cohen-random} introduced the definition of a pathway to denote an object that, according to our definition, we would call a $\langle\mathfrak d,\mathscr F,\kappa\rangle$-pathway, where $\mathscr F$ is the family containing the join (binary) operation on $\omega^\omega$, along with a unary operation $F_T:\omega^\omega\longrightarrow\omega^\omega$ for each algorithm $T$ that uses an oracle ($F_T(f)$ is supposed to be the function mapping each $n$ to the output of the algorithm $T$, which can use $f$ as an oracle, when given input $n$). Then~\cite[Theorem 1.1]{cohen-random} is the statement that the existence of such a pathway implies the existence of a P-point, whereas~\cite[Theorem 2.5]{cohen-random} asserts that there is such a pathway, of length $\omega_1$, in the Random model. As we mentioned above, there is a serious gap in the proof of the second statement, but not on the proof of the first one\footnote{The most up-to-date version of~\cite{david-osvaldo} contains an appendix with a detailed explanation of the gap in this proof.}. In the following subsection we provide such a proof (for the convenience of the reader) of this fact, from a pathway with a very minimal (in fact, finite) set of operations $\mathscr F$.

\subsection{P-points from a $\mathfrak d$-pathway}

We will describe the family of finitary operations $\mathscr F$ for which we need only a $\mathfrak d$-pathway in order to construct a P-point. We start by fixing a coding function
\begin{equation*}
\cd:\omega^\omega\longrightarrow\mathfrak P(\omega)^{\downarrow\omega},
\end{equation*}
where $\mathfrak P(\omega)^{\downarrow\omega}$ denotes the set of all decreasing sequences $\langle X_n\big|n<\omega\rangle\in\mathfrak P(\omega)^\omega$ such that $\bigcap_{n<\omega}X_n=\varnothing$; and we fix also a decoding function
\begin{equation*}
\dcd:\mathfrak P(\omega)^{\downarrow\omega}\longrightarrow\omega^\omega
\end{equation*}
such that $\cd\circ\dcd=\id_{\mathfrak P(\omega)^{\downarrow\omega}}$ (in particular, $\cd$ is surjective and $\dcd$ is injective). The specific nature of the functions $\cd$ and $\dcd$ will not be important, although later on we will see the convenience of these functions to be set-theoretically definable and absolute. For now, just for the sake of ensuring that there is at least one pair of functions with the desired properties, we point out that it is possible to let
\begin{equation*}
\cd(f)=\langle\{k<\omega\big|n<f(k)\}\big|n<\omega\rangle,
\end{equation*}
and
\begin{eqnarray*}
\dcd(\langle X_n\big|n<\omega\rangle)=\sum_{n<\omega}\chi_{X_n},
\end{eqnarray*}
where $\chi_X$ is just the characteristic function of $X\subseteq\omega$ (note that, since the $X_n$ have empty intersection, the value of $\sum_{n<\omega}\chi_{X_n}(k)$ is always finite). As we just said, we will later on use the fact that these particular definitions of the functions $\cd$ and $\dcd$ are set-theoretically definable and absolute, but in this section we do not really need to know this fact nor the particular incarnation of the functions under use.

Now suppose that we have an element $\vec{X}=\langle X_n\big|n<\omega\rangle\in\mathfrak P(\omega)^{\downarrow\omega}$, and a function $f:\omega\longrightarrow\omega$. We define the pseudointersection of $\vec{X}$ with growth controlled by $f$ to be the set
\begin{equation*}
\ps(\vec{X},f)=\bigcup_{n<\omega}X_n\cap f(n)
\end{equation*}
(notice that $\ps(\vec{X},f)$ is indeed a pseudointersection for the $X_n$, and a subset of $X_0$), and we also define the ``fast-enough-growing function'' $f_{\vec{X}}$ of the sequence $\vec{X}$ by the recursive definition
\begin{equation*}
\begin{cases}
f_{\vec{X}}(0)=\min\{k<\omega\big|k\in X_0\}+1; \\
f_{\vec{X}}(n+1)=\min\{k<\omega\big|k>f_{\vec{X}}(n)\wedge k\in X_{n+1}\}+1
\end{cases}
\end{equation*}
(notice that, whenever $f\not\leq^*f_{\vec{X}}$, then $\ps(\vec{X},f)$ will be an infinite set). For notational convenience, we define functions $e_n:\mathfrak P(\omega)^{\downarrow\omega}\longrightarrow\mathfrak P(\omega)$ by $e_n(\langle X_n\big|n<\omega\rangle)=X_n$.

\begin{definition}
We define the following finitary operations on $\omega^\omega$:
\begin{itemize}
\item Unary $F_c:\omega^\omega\longrightarrow\omega^\omega$ given by 
\begin{equation*}
F_c(h)=\dcd(\langle(\omega\setminus e_0(\cd(h)))\setminus n\big|n<\omega\rangle),
\end{equation*}
\item Unary $F_f:\omega^\omega\longrightarrow\omega^\omega$ given by 
\begin{equation*}
F_f(h)=f_{\cd(h)},
\end{equation*}
\item Unary $F_s:\omega^\omega\longrightarrow\omega^\omega$ given by 
\begin{equation*}
F_s(h)=\dcd(\langle e_0(\cd(h))\setminus n\big|n<\omega\rangle)
\end{equation*}
\item Binary $F_p:\omega^\omega\times\omega^\omega\longrightarrow\omega^\omega$ given by 
\begin{equation*}
F_p(g,h)=\dcd(\langle\ps(\cd(g),h)\setminus n\big|n<\omega\rangle),
\end{equation*}
\item Binary $F_i:\omega^\omega\times\omega^\omega\longrightarrow\omega^\omega$ given by 
\begin{equation*}
F_i(g,h)=\dcd(\langle e_n(\cd(g))\cap e_n(\cd(h))\big|n<\omega\rangle).
\end{equation*}
\end{itemize}
\end{definition}

\begin{theorem}\label{existsppoint}
Let $\mathscr F_{\text{P-pt}}=\{F_c,F_f,F_s,F_i,F_p\}$. If there is a $\langle\mathfrak d,\mathscr F_{\text{P-pt}},\kappa\rangle$-pathway, then there is a P-point.
\end{theorem}

\begin{proof}
Let $\langle B_\alpha\big|\alpha<\kappa\rangle$ be a $\langle\mathfrak d,\mathscr F_{\text{P-pt}},\kappa\rangle$-pathway. By Observation~\ref{strongercondition}, we can assume without loss of generality that for each $\alpha<\kappa$, there is $f_\alpha\in B_{\alpha+1}$ such that $(\forall g\in B_\alpha)(f_\alpha\not\leq^* g)$. Moreover we can assume that $\cf(\kappa)=\kappa\geq\mathfrak b>\omega$. We define auxiliary sequences of sets $\langle C_\alpha\big|\alpha<\kappa\rangle$ and $\langle D_\alpha\big|\alpha<\kappa\rangle$ by $C_\alpha=\cd[B_\alpha]$ and $D_\alpha=e_0[C_\alpha]$. Then clearly both the $C_\alpha$ and the $D_\alpha$ will be continuous increasing sequences, and (since $\cd$ is surjective) $\bigcup_{\alpha<\kappa}C_\alpha=\mathfrak P(\omega)^{\downarrow\omega}$, and also (for obvious reasons) $\bigcup_{\alpha<\kappa}D_\alpha=\mathfrak P(\omega)$. Moreover, it follows from the fact that the $B_\alpha$ are closed under $F_c$ and $F_i$, that the $D_\alpha$ will be closed under taking complements and intersections, and moreover the $C_\alpha$ will be closed under the operation taking two decreasing sequences $\vec{X},\vec{Y}$ and outputting $\langle X_n\cap Y_n\big|n<\omega\rangle$. Closure under $F_f$ just means that if $\vec{X}\in C_\alpha$ then $f_{\vec{X}}\in B_\alpha$; closure under $F_p$ means that if $h\in B_\alpha$ and $\vec{X}\in C_\alpha$, then $\ps(\vec{X},h)\in D_\alpha$ and the sequence of ``tail ends'' $\langle\ps(\vec{X},h)\setminus n\big|n<\omega\rangle\in C_\alpha$; and closure under $F_s$ means that whenever $X\in D_\alpha$, the sequence $\langle X\setminus n\big|n<\omega\rangle\in C_\alpha$ (all of the facts from this and the previous sentences use the property that $\cd\circ\dcd=\id_{\mathfrak P(\omega)^{\downarrow\omega}}$).

Recursively construct a continuous increasing sequence of filters $\langle\mathcal F_\alpha\big|\alpha<\kappa\rangle$ satisfying the following three conditions:
\begin{enumerate}
\item Each $\mathcal F_\alpha$ has a basis of sets that belong to $D_\alpha$,
\item for each $X\in D_\alpha$, either $X\in\mathcal F_{\alpha+1}$ or $\omega\setminus X\in\mathcal F_{\alpha+1}$,
\item for each $\vec{X}\in C_\alpha$ satisfying $(\forall n<\omega)(e_n(\vec{X})\in\mathcal F_\alpha)$, there exists a pseudointersection $X\in D_{\alpha+1}$ such that $X\in\mathcal F_{\alpha+1}$.
\end{enumerate}
If we succeed in such a construction, letting $u=\bigcup_{\alpha<\kappa}\mathcal F_\alpha$ will yield an ultrafilter (by condition (2)) which is a P-point (by condition (3), together with the fact that $\cf(\kappa)>\omega$).

We now proceed to carry on the construction, which goes as follows. Without loss of generality (since $\cf(\kappa)=\kappa>\omega$), all sets of the form $\omega\setminus n$, for $n<\omega$, are elements of $D_0$, so we start by letting $\mathcal F_0=\{\omega\setminus n\big|n<\omega\}$. We know that at at limit stages we will just take unions, so assume that we are given $\mathcal F_\alpha$ with a basis of sets that belong to $D_\alpha$, and let us explain the construction of $\mathcal F_{\alpha+1}$. We start by picking an ultrafilter $\mathcal U\supseteq\mathcal F_\alpha$, and we let $\mathcal F_{\alpha+1}$ be the filter generated by the family
\begin{equation*}
\mathcal B_\alpha=\{\ps(\vec{X},f_\alpha)\big|\vec{X}\in C_\alpha\wedge(\forall n<\omega)(e_n(\vec{X})\in\mathcal U)\}.
\end{equation*}
Whenever $\vec{X}\in C_\alpha$, we have $f_{\vec{X}}\in B_\alpha$ and so $f_\alpha\not\leq^*f_{\vec{X}}$, which implies that $\ps(\vec{X},f_\alpha)$ is infinite, and moreover since $\vec{X}\in C_{\alpha+1}$ and $f_\alpha\in B_{\alpha+1}$, we also have that $\ps(\vec{X},f_\alpha)\in D_{\alpha+1}$. Thus $\mathcal B_\alpha\subseteq D_{\alpha+1}$, let us proceed to verify that $\mathcal B_\alpha$ is indeed a filter base. If $\vec{X}_1,\ldots,\vec{X}_k\in C_\alpha$ are such that $(\forall i\leq k)(\forall n<\omega)(e_n(\vec{X}_i)\in\mathcal U)$, then we will have that the sequence
\begin{equation*}
\vec{Y}=\langle e_n(\vec{X}_1)\cap\cdots\cap e_n(\vec{X}_k)\big|n<\omega\rangle\in C_\alpha,
\end{equation*}
and clearly $(\forall n<\omega)(e_n(\vec{Y})\in\mathcal U)$; thus $\ps(\vec{Y},f_\alpha)\in\mathcal B_\alpha$ and it is not hard to see that $\ps(\vec{Y},f_\alpha)\subseteq\ps(\vec{X}_i,f_\alpha)$ for all $i\leq k$. Thus $\mathcal B_\alpha$ is a filterbase of elements of $D_\alpha$, which means that $\mathcal F_{\alpha+1}$ satisfies condition (1).

To see that $\mathcal F_{\alpha+1}$ satisfies condition (2), let $X\in D_\alpha$. Since $\mathcal U$ is an ultrafilter, there is $Y\in\{X,\omega\setminus X\}$ such that $Y\in\mathcal U$, and since $D_\alpha$ is closed under taking complements, we have that $Y\in D_\alpha$. Then our closure assumptions also imply that $\vec{Y}=\langle Y\setminus n\big|n<\omega\rangle\in C_\alpha$, and clearly $(\forall n<\omega)(e_n(\vec{Y})=Y\cap(\omega\setminus n)\in\mathcal U)$, thus $\mathcal B_\alpha$ contains the element $\ps(\vec{Y},f_\alpha)\subseteq Y$ which means that $Y\in\mathcal F_{\alpha+1}$.

Finally, to check condition (3), suppose that $\vec{X}$ is such that $(\forall n<\omega)(e_n(\vec{X})\in\mathcal F_\alpha\subseteq\mathcal U)$. Then $\ps(\vec{X},f_\alpha)\in\mathcal B_\alpha\subseteq\mathcal F_{\alpha+1}$, and we are done.

\end{proof}

\subsection{Gruff ultrafilters from a $\mathfrak d$-pathway}

Now we turn our attention to gruff ultrafilters. Recall that an ultrafilter $u$ on $\mathbb Q$ is said to be \textbf{gruff} if it has a base of perfect subsets of $\mathbb Q$, that is, if $(\forall A\in u)(\exists P\in u)(P\text{ is perfect and }P\subseteq A)$, where perfect just means closed and crowded in the usual Euclidean topology that $\mathbb Q$ inherits from $\mathbb R$. We can add the requirement that the sets generating the relevant ultrafilter are unbounded in $\mathbb Q$, in addition to just perfect, and the question of whether such a gruff ultrafilter exists is still equivalent to the analogous question without this extra requirement.

We will describe the family of finitary operations for which we need a $\mathfrak d$-pathway in order to construct a gruff ultrafilter. We start by fixing a coding function
\begin{equation*}
\cd:\omega^\omega\longrightarrow\mathfrak P(\mathbb Q),
\end{equation*}
and we fix also a decoding function
\begin{equation*}
\dcd:\mathfrak P(\mathbb Q)\longrightarrow\omega^\omega
\end{equation*}
such that $\cd\circ\dcd=\id_{\mathfrak P(\mathbb Q)}$ (in particular, $\cd$ is surjective and $\dcd$ is injective). The specific nature of the functions $\cd$ and $\dcd$ will not be important, although later on we will see the convenience of these functions to be set-theoretically definable and absolute. For now, just for the sake of ensuring that there is at least one pair of functions with the desired properties, we point out that it is possible, given an effective injective enumeration $\langle q_n\big|n<\omega\rangle$ of $\mathbb Q$, to let
\begin{equation*}
\cd(f)=\{q_n\in\mathbb Q\big|f(n)\neq 0\},
\end{equation*}
and
\begin{eqnarray*}
\dcd(X)=\chi_{\{n<\omega\big|q_n\in X\}},
\end{eqnarray*}
(where $\chi_Y$ is just the characteristic function of $Y\subseteq\omega$). As we just said, we will later on use the fact that these particular definitions of the functions $\cd$ and $\dcd$ are set-theoretically definable and absolute, but in this section we don't really need to know neither this fact nor the particular incarnation of the functions under use.

Now suppose that we have a subset $X\subseteq\mathbb Q$ and a function $f:\omega\longrightarrow\omega$. We define the shrinking of $X$ controlled by $f$ to be the set
\begin{equation*}
\shr(X,f)=\mathbb Q\setminus\left(\bigcup_{q_n\notin X}J_n^f\right)\subseteq X,
\end{equation*}
defined using the auxiliary intervals $J_n^f$ given by
\begin{equation*}
J_n^f=\left(q_n-\frac{\sqrt{2}}{k},q_n+\frac{\sqrt{2}}{k}\right),
\end{equation*}
where $k$ is the least possible natural number that ensures $q_m\notin J_n^f$ for every $n\neq m\leq f(n)$. Note that by definition $\shr(X,f)$ is a closed subset of $\mathbb Q$. %Note that if $X\subseteq Y$, then $X(f)\subseteq Y(f)$; note also 
%that $X(f)\cap Y(f)=(X\cap Y)(f)$ for all $X,Y\subseteq\mathbb Q$, and a similar statement holds for intersections of any finite number of sets. Furthermore, if $g\leq f$ then $X(g)\subseteq X(f)$.
%\item If $g\leq^*f$ then there is an $N\in\mathbb N$ such that 
% $X(g)\setminus(0,N)\subseteq X(f)\setminus (0,N)$.

%\item If $f$ is unbounded, then every rational number belongs to 
 %only finitely many of the $J_n^f$. In fact, if $f$ is strictly increasing then 
 %the rational number $q_n$ can only belong to at most $n$ of the $J_m^f$, since 
 %in this case, whenever $m\gt n$ we have that $n\leq f(n)\lt f(m)$ and therefore 
 %$q_n\notin J_m^f$.

We also define the ``fast-enough-growing function'' $f_X$ of the subset $X\subseteq\mathbb Q$ by letting $f_X$ be constantly zero if $X$ is not crowded unbounded, and otherwise letting $f_X$ be given by the recursive definition
\begin{eqnarray*}
f_X(0)=\min\{k<\omega & \big| & q_k\in X\}; \\
f_X(n+1)=\min\{k<\omega & \big| & k>f_X(n) \\
 & & \wedge(\forall i\leq n)(q_i\in X\Rightarrow(\exists j\leq k)(j>f_X(n)\wedge \\
 & & q_j\in X\wedge|q_i-q_j|<\frac{1}{2^n}\wedge q_j\notin\bigcup_{i<l\leq n}J_l^{\id})) \\
 & & \wedge(\exists i\leq k)(i>f_X(n)\wedge \\
 & & q_i\in X\wedge q_i>n+1\wedge q_i\notin\bigcup_{l\leq n}J_l^{\id})\}.
\end{eqnarray*}
In~\cite[Lemma 3.1]{gruffpaper} it is proved that, if $X\subseteq\mathbb Q$ is a crowded unbounded set, then the function $f_X$ defined as above has the property that, whenever $g:\omega\longrightarrow\omega$ is increasing and $g\not\leq^*f_X$, then $\shr(X,g)$ will be a perfect unbounded set.

In the following definition, given a set $X\subseteq\mathbb Q$ we will denote by $M(X)$ the maximal crowded unbounded subset of $X$.

\begin{definition}
We define the following finitary operations on $\omega^\omega$:
\begin{itemize}
\item Unary $F_c:\omega^\omega\longrightarrow\omega^\omega$ given by 
\begin{equation*}
F_c(h)=\dcd(\mathbb Q\setminus\cd(h)),
\end{equation*}
\item Unary $F_m:\omega^\omega\longrightarrow\omega^\omega$ given by 
\begin{equation*}
F_m(h)=\dcd(M(\cd(h)))
\end{equation*}
\item Unary $F_f:\omega^\omega\longrightarrow\omega^\omega$ given by 
\begin{equation*}
F_f(h)=f_{\cd(h)},
\end{equation*}
\item Binary $F_s:\omega^\omega\times\omega^\omega\longrightarrow\omega^\omega$ given by 
\begin{equation*}
F_s(g,h)=\dcd(\shr(\cd(g),h)),
\end{equation*}
\item Binary $F_i:\omega^\omega\times\omega^\omega\longrightarrow\omega^\omega$ given by 
\begin{equation*}
F_i(g,h)=\dcd(\cd(g)\cap\cd(h)).
\end{equation*}
\end{itemize}
\end{definition}

\begin{theorem}\label{existsgruff}
Let $\mathscr F_{\text{gruff}}=\{F_c,F_m,F_f,F_s,F_i\}$. If there is a $\langle\mathfrak d,\mathscr F_{\text{gruff}},\kappa\rangle$-pathway, then there is a gruff ultrafilter.
\end{theorem}

\begin{proof}
Let $\langle B_\alpha\big|\alpha<\kappa\rangle$ be a $\langle\mathfrak d,\mathscr F_{\text{gruff}},\kappa\rangle$-pathway. By Observation~\ref{strongercondition}, we can assume without loss of generality that for each $\alpha<\kappa$, there is $f_\alpha\in B_{\alpha+1}$ such that $(\forall g\in B_\alpha)(f_\alpha\not\leq^* g)$. Moreover we can assume that $\cf(\kappa)=\kappa\geq\mathfrak b>\omega$. We define an auxiliary sequence $\langle C_\alpha\big|\alpha<\kappa\rangle$ by $C_\alpha=\cd[B_\alpha]$. Then the sequence of the $C_\alpha$ will be continuous increasing, and (since $\cd$ is surjective) $\bigcup_{\alpha<\kappa}C_\alpha=\mathfrak P(\mathbb Q)$. Furthermore, it follows from the fact that the $B_\alpha$ are closed under $F_c$, $F_m$ and $F_i$, that the $C_\alpha$ will be closed under taking complements, intersections, and maximal crowded subsets (all of these use the fact that $\cd\circ\dcd=\id_{\mathfrak P(\mathbb Q)}$). Moreover, since $B_\alpha$ is closed under $F_s$, then whenever $X\in C_\alpha$ and $f\in B_\alpha$ we will have $\shr(X,f)\in C_\alpha$; and since $B_\alpha$ is closed under $F_f$, then whenever $X\in C_\alpha$ we will have that $f_X\in B_\alpha$.

Recursively construct a continuous increasing sequence of filters $\langle\mathcal F_\alpha\big|\alpha<\kappa\rangle$ satisfying the following two conditions:
\begin{enumerate}
\item Each $\mathcal F_\alpha$ has a basis of perfect unbounded sets that belong to $C_\alpha$,
\item for each $X\in C_\alpha$, either $X\in\mathcal F_{\alpha+1}$ or $\mathbb Q\setminus X\in\mathcal F_{\alpha+1}$.
\end{enumerate}
If we succeed in such a construction, letting $u=\bigcup_{\alpha<\kappa}\mathcal F_\alpha$ will yield a gruff (by condition (1)) ultrafilter (by condition (2)).

We now proceed to carry on the construction, which goes as follows. Without loss of generality (since $\cf(\kappa)=\kappa>\omega$), all sets of the form $\mathbb Q\setminus(-\infty,n]$, for $n<\omega$, are elements of $C_0$, so we start by letting $\mathcal F_0=\{\mathbb Q\setminus(-\infty,n]\big|n<\omega\}$. We know that at at limit stages we will just take unions, so assume that we are given $\mathcal F_\alpha$ with a basis of perfect unbounded sets that belong to $C_\alpha$, and let us explain the construction of $\mathcal F_{\alpha+1}$. We start by picking an ultrafilter $\mathcal U\supseteq\mathcal F_\alpha$, all of whose elements contain a crowded unbounded set (this can be done since the family of crowded unbounded subsets of $\mathbb Q$ generates a coideal), and then we let $\mathcal F_{\alpha+1}$ be the filter generated by the family
\begin{equation*}
\mathcal B_\alpha=\{\shr(X,f_\alpha)\big|X\in C_\alpha\cap\mathcal U)\}.
\end{equation*}

For any $X\in C_\alpha$, we have $f_X\in B_\alpha$ as well and so $f_\alpha\not\leq^*f_X$, which implies that $\shr(X,f_\alpha)$ is perfect and unbounded, and moreover since $X\in C_{\alpha+1}$ and $f_\alpha\in B_{\alpha+1}$, we also have that $\shr(X,f_\alpha)\in C_{\alpha+1}$. Thus $\mathcal B_\alpha\subseteq C_{\alpha+1}$, let us proceed to verify that $\mathcal B_\alpha$ is indeed a filterbase. If $X_1,\ldots,X_k\in C_\alpha\cap\mathcal U$, then we will have that $X=M(X_1\cap\cdots\cap X_k)\in C_\alpha\cap\mathcal U$, and therefore $\shr(X,f_\alpha)\in\mathcal B_\alpha$. It is not hard to see that 
$\shr(X,f_\alpha)\subseteq\shr(X_1\cap\cdots\cap X_n,f_\alpha)=\shr(X_1,f_\alpha)\cap\cdots\cap\shr(X_k,f_\alpha)$. Thus $\mathcal B_\alpha$ is a filterbase of perfect unbounded subsets of $\mathbb Q$ that belong to $C_\alpha$, which means that $\mathcal F_{\alpha+1}$ satisfies condition (1).

To see that $\mathcal F_{\alpha+1}$ satisfies condition (2), let $X\in C_\alpha$. Since $\mathcal U$ is an ultrafilter, there is $Y\in\{X,\mathbb Q\setminus X\}$ such that $Y\in\mathcal U$, and since $C_\alpha$ is closed under taking complements, we have that $Y\in C_\alpha$. Thus $Y\in C_\alpha\cap\mathcal U$, which implies that $\shr(Y,f_\alpha)\in\mathcal B_\alpha$ and therefore either $X$ or $\mathbb Q\setminus X$ will belong to $\mathcal F_{\alpha+1}$. This establishes condition condition (2), and we are done.
\end{proof}

The results developed in this section allow us to recover two previously known theorems, whose proofs seemed to be different earlier on but we can now see that they are both instances of a more general phenomenon.

\begin{corollary}
If $\mathfrak d=\mathfrak c$, then for every finite family $\mathscr F$ of finitary operations on $\omega^\omega$, there exists a $\langle\mathfrak b,\mathscr F,\mathfrak c\rangle$-pathway. Consequently, if $\mathfrak d=\mathfrak c$ we get a new proof that there are both P-points and gruff ultrafilters.
\end{corollary}

\begin{proof}
The existence of the relevant pathway from $\mathfrak d=\mathfrak c$ follows from Theorem~\ref{characteristiclarge}. Such an existential statement yields a P-point by means of Theorem~\ref{existsppoint}, and a gruff ultrafilter by Theorem~\ref{existsgruff}.
\end{proof}

\section{Pathways for the cardinal characteristic $\covm$}

Similar to the results on the previous section about consequences of a pathway for $\mathfrak d$, we will now analyse how some of the consequences of $\covm=\mathfrak c$ are actually consequences of the existence of pathways for $\covm$. The first thing to notice, is that thinking of $\covm$ as ``the covering number for the $\sigma$-ideal of meagre sets'' is not a very fruitful viewpoint when it comes to pathways, so we will define this cardinal characteristic slightly differently. We start by noting that the least number of meagre sets needed to cover the real line is the same as the least number of closed nowhere dense sets needed to cover the real line. The latter cardinal characteristic can be represented as $\langle 2^\omega,\mathrm{NWD},\in\rangle$, where $\mathrm{NWD}$ denotes the set of closed nowhere dense subsets of $2^\omega$. The mapping $:\mathrm{NWD}\longrightarrow\mathrm{OD}$ (where $\mathrm{OD}$ denotes the set of open dense subsets of $2^\omega$) taking every closed nowhere dense set to its complement clearly defines an invertible morphism between $\langle 2^\omega,\mathrm{NWD},\in\rangle$ and $\langle 2^\omega,\mathrm{OD},\notin\rangle$. Since open dense subsets of $2^\omega$ correspond naturally to dense subsets of the Cohen forcing notion $2^{<\omega}$, this chain of reasoning leads us to state the following definition.

\begin{definition}
Consider the set $2^{<\omega}$, partially ordered by reverse inclusion, and let $\mathscr D$ be the set of all dense subsets of this partial order. Let $E\subseteq 2^\omega\times\mathscr D$ be the relation given by $x\ E\ D$ if and only if the filter $\{x\upharpoonright n\big|n<\omega\}$ does not intersect the dense set $D$. Throughout this note, our ``official'' definition of the cardinal characteristic $\covm$ will be that it is the evaluation of the triple $\langle 2^\omega,\mathscr D,E\rangle$ just described.
\end{definition}

Thus a pathway for $\covm$ (of length $\kappa$) consists of a continuous increasing sequence $\langle\mathscr D_\alpha\big|\alpha<\kappa\rangle$, where each $\mathscr D_\alpha$ is a family of dense sets closed under certain finitary operations, such that $\bigcup_{\alpha<\kappa}\mathscr D_\alpha=\mathscr D$, and for every $\alpha<\kappa$ there is a $\mathscr D_\alpha$-generic real $x_\alpha$ (that is, an $x_\alpha\in 2^\omega$ such that $(\forall D\in\mathscr D_\alpha)(\{x_\alpha\upharpoonright n\big|n<\omega\}\cap D\neq\varnothing)$.

We proceed to show that from an appropriate such pathway we can construct a Q-point\footnote{I am, in fact, fairly certain that we can also construct a selective ultrafilter from a $\covm$-pathway. On the other hand, I suspect, but have not yet been able to verify, that such a pathway would also entail the existence of stable ordered union ultrafilters.}. We will describe the operations needed for this. We start by fixing a function $\cd:\mathscr D\longrightarrow\mathfrak P(\omega)$, along with a decoding function $\dcd:\mathfrak P(\omega)\longrightarrow\mathscr D$ in such a way that $\cd\circ\dcd=\id\upharpoonright\mathfrak P(\omega)$ (in particular, $\cd$ is surjective and $\dcd$ is injective). The specific nature of the functions $\cd$ and $\dcd$ will not be important, although later on we will see the convenience of these functions to be set-theoretically definable and absolute (I'm hoping for it to be true that such functions can be chosen to be both set-theoretically definable and absolute). We also fix a bijection $\inpa:\mathfrak P(\omega)\longrightarrow\mathscr I$, where $\mathscr I$ is the family of all partitions of $\omega$ into finite intervals, by letting $\inpa(X)=\{[0,x_0]\}\cup\{(x_n,x_{n+1}]\big|n<\omega\}$, where $\langle x_n\big|n<\omega\rangle$ is the increasing enumeration of $X$.

We now will adapt our Cohen reals so that they yield specific subsets of given sets that are selectors of a given partition into intervals. So let $X\subseteq\omega$, let $I=\{I_n\big|n<\omega\}\in\mathscr I$ (we assume that the enumeration of the intervals in $I$ is appropriately increasing), and let $x\in 2^\omega$. We define the \textbf{Cohen real} coded by $x$ and associated to $X$ and $I$ to be the set $\coh(X,I,x)=\{x_n\big|n<\omega\}$ chosen recursively by $x_0=\min\{k\in X\big|x(k)=1\}$, and
\begin{eqnarray*}
x_{n+1}=\min\{k\in X\big|x(k)=1\wedge(\forall i,j<\omega)((k\in I_i\wedge x_n\in I_j)\Rightarrow j+1<i)\}.
\end{eqnarray*}
It is clear that $\coh(X,I,x)\subseteq X$ and that $(\forall n<\omega)(|\coh(X,I,x)\cap I_n|\leq 1)$.

We also want to define appropriate dense sets. For each finite collection of subsets $X_1,\ldots,X_n\subseteq\omega$, along with each collection $I_1,\ldots,I_n\in\mathscr I$ of as many partitions into intervals, as well as each $m<\omega$, define the set $D(X_1,\ldots,X_n;I_1,\ldots,I_n;m)$ to be the subset of $2^{<\omega}$ consisting of all conditions $s$ such that 
\begin{equation*}s\Vdash``|\coh(X_1,I_1,\mathring{x})\cap\cdots\cap\coh(X_n,I_n,\mathring{x})|\geq m",
\end{equation*}
where $\mathring{x}$ is the $2^{<\omega}$-name for the generic Cohen real. It is easy to check that, if $X_1\cap\cdots\cap X_n$ is infinite, then $D(X_1,\ldots,X_n;I_1,\ldots,I_n;m)$ is in fact a dense set in the Cohen forcing.

\begin{definition}
We define the following finitary operations on $\mathscr D$:
\begin{itemize}
\item Unary $F_c:\mathscr D\longrightarrow\mathscr D$ given by 
\begin{equation*}
F_c(D)=\dcd(\omega\setminus\cd(D)),
\end{equation*}
\item Binary $F_i:\mathscr D\times\mathscr D\longrightarrow\mathscr D$, given by 
\begin{equation*}
F_i(D,D')=\dcd(\cd(D)\cap\cd(D')).
\end{equation*}
\item Ternary $F_r:\mathscr D\times\mathscr D\times\mathscr D\longrightarrow\mathscr D$ given by 
\begin{equation*}
F_s(D,D',D'')=\dcd(\coh(\cd(D),\inpa(\cd(D')),\chi_{\cd(D'')})),
\end{equation*}

\item For each $n,m<\omega$, we define a $2n$-ary operation $F_d^{n,m}:(\mathscr D)^{2n}\longrightarrow\mathscr D$, where $F_d^{n,m}(D_1,\ldots,D_{2n})$ is given by 
\begin{equation*}
D(\cd(D_1),\ldots,\cd(D_n);\inpa(\cd(D_{n+1})),\ldots,\inpa(\cd(D_{2n}));m).
\end{equation*}

\end{itemize}
\end{definition}

\begin{theorem}\label{existsqpt}
Let $\mathscr F_{\text{Q-pt}}=\{F_c,F_i,F_r\}\cup\{F_d^{n,m}\big|n,m<\omega\}$. Under the assumption that there is a $\langle\covm,\mathscr F_{\text{Q-pt}},\kappa\rangle$-pathway, there is a Q-point\footnote{For the first time, we actually seem to need closure under infinitely many operations!}.
\end{theorem}

\begin{proof}
Let $\langle\mathscr D_\alpha\big|\alpha<\kappa\rangle$ be a $\langle\covm,\mathscr F_{\text{Q-pt}},\kappa\rangle$-pathway. By thinning out the sequence of dense sets if necessary, we may assume that for each $\alpha<\kappa$, the $x_\alpha$ such that $(\forall D\in\mathscr D_\alpha)(\{x_\alpha\upharpoonright n\big|n<\omega\}\cap D\neq\varnothing)$ satisfies that $X_\alpha=\{n<\omega\big|x_\alpha(n)=1\}\in\cd[\mathscr D_{\alpha+1}]$. We define an auxiliary sequence $\langle A_\alpha\big|\alpha<\kappa\rangle$ by $A_\alpha=\cd[\mathscr D_\alpha]$. Then the sequence of the $A_\alpha$ will be continuous increasing, and (since $\cd$ is surjective) $\bigcup_{\alpha<\kappa}A_\alpha=\mathfrak P(\omega)$. Furthermore, it follows from the fact that the $\mathscr D_\alpha$ are closed under $F_c$ and $F_i$, that the $A_\alpha$ will be closed under taking complements and intersections. Since the $\mathscr D_\alpha$ are closed under $F_r$, then the $A_\alpha$ will be closed under computing Cohen reals (that is, if $X,Y,Z\in A_\alpha$ then $\coh(X,\inpa(Y),\chi_Z)\in A_\alpha$). And since the $\mathscr D_\alpha$ are closed under all of the $F_d^{n,m}$, then whenever we have any collection $X_1,\ldots,X_{2n}$ of even many elements of $A_\alpha$, the dense set $D(X_1,\ldots,X_n;\inpa(X_{n+1}),\ldots,\inpa(X_{2n});m)\in\mathscr D_\alpha$.

Recursively construct a continuous increasing sequence of filters $\langle\mathcal F_\alpha\big|\alpha<\kappa\rangle$ satisfying the following three conditions:
\begin{enumerate}
\item Each $\mathcal F_\alpha$ has a basis of sets that belong to $A_\alpha$,
\item for each $X\in A_\alpha$, either $X\in\mathcal F_{\alpha+1}$ or $\mathbb Q\setminus X\in\mathcal F_{\alpha+1}$,
\item for each $X\in A­_\alpha$, there exists a $Y\in\mathcal F_{\alpha+1}$ such that $(\forall n<\omega)(|Y\cap I_n|\leq 1)$, where $\{I_n\big|n<\omega\}=\inpa(X)$.
\end{enumerate}
If we succeed in such a construction, letting $u=\bigcup_{\alpha<\kappa}\mathcal F_\alpha$ will yield an ultrafilter (by condition (2)) which is a Q-point (by condition (3)).

We start by letting $\mathcal F_0=\{\omega\}$. We know that at at limit stages we will just take unions, so assume that we are given $\mathcal F_\alpha$ with a basis of sets that belong to $A_\alpha$, and let us explain the construction of $\mathcal F_{\alpha+1}$. We start by picking an ultrafilter $\mathcal U\supseteq\mathcal F_\alpha$, and then we let $\mathcal F_{\alpha+1}$ be the filter generated by the subbasis
\begin{equation*}
\mathcal B_\alpha=\{\coh(X,\inpa(Y),x_\alpha)\big|X,Y\in A_\alpha\wedge X\in\mathcal U)\}.
\end{equation*}
Clearly every element $\coh(X,\inpa(Y),x_\alpha)\in\mathcal B_\alpha$ belongs to $A_{\alpha+1}$ (since $X,Y\in A_\alpha\subseteq A_{\alpha+1}$ and $X_\alpha\in A_{\alpha+1}$). Also, if $X\in\mathcal F_\alpha$ then $X\in \mathcal U$ and therefore $\mathcal F_
{\alpha+1}\ni\coh(X,\inpa(X),x_\alpha)\subseteq X$, thus $\mathcal F_{\alpha+1}$ really extends $\mathcal F_\alpha$. Now, since $A_{\alpha+1}$ is closed under intersections, we only need to show that each intersection of finitely many elements of $\mathcal B_\alpha$ is infinite, to get that $\mathcal F_{\alpha+1}$ satisfies condition (1). To see this, let $X_1,\ldots,X_n\in A_\alpha\cap\mathcal U$ and $Y_1,\ldots,Y_n\in A_\alpha$ determine $n$ elements $\coh(X_i,\inpa(Y_i),x_\alpha)\in\mathcal B_\alpha$. Then for each $m$ we must have that the dense sets 
$D(X_1,\ldots,X_n;Y_1,\ldots,Y_n;m)\in\mathscr D_\alpha$ and therefore $\vec{x_\alpha}$ must intersect all of those dense sets, which clearly implies that $\bigcap_{1\leq i\leq n}\coh(X_i,\inpa(Y_i),x_\alpha)$ is an infinite set.

To see that $\mathcal F_{\alpha+1}$ satisfies condition (2), let $X\in A_\alpha$. Then there is $Y\in\{X,\omega\setminus X\}$ such that $Y\in\mathcal U$, the fact that $A_\alpha$ is closed under complements implies that $Y\in A_\alpha$, and thus $Y\supseteq\coh(Y,\inpa(Y),x_\alpha)\in\mathcal F_{\alpha+1}$. Now for property (3), given $X\in A_\alpha$, note that letting $Y=\coh(\omega,\inpa(X),x_\alpha)$ yields a witness of such property.
\end{proof}

I will now sketch the argument that leads me to believe that one can build a selective ultrafilter from an appropriate $\covm$-pathway; I promise to work out the details in a future version of this document. The idea is that a Cohen real is unbounded, so the existence of a $\covm$-pathway should imply the existence of a $\mathfrak d$-pathway (by appropriately adjusting the finitary operations under consideration); amalgamating such an argument with our proof of Theorem~\ref{existsppoint}, one should be able to conclude that, for a suitable family $\mathscr F_{\text{P-pt}}'$ of finitary operations, the existence of a $\langle\covm,\mathscr F_{\text{P-pt}}',\kappa\rangle$-pathway implies the existence of a P-point. Now, by alternating the steps followed in the proof of Theorem~\ref{existsppoint} with those from the proof of Theorem~\ref{existsqpt} (say, by doing the former in even stages and the latter in odd stages of the recursive construction), it looks like a very reasonable conjecture that the existence of a $\langle\covm,\mathscr F_{\text{Q-pt}}\cup\mathscr F_{\text{P-pt}}',\kappa\rangle$-pathway should imply the existence of an ultrafilter which is both a Q-point and a P-point, in other words, a selective ultrafilter.

\section{Morphisms between cardinal characteristics and their effect on pathways}

We now turn to results relating the existence of pathways for different kinds of cardinal characteristics. Recall that, if $\mathfrak x=\langle A,B,E\rangle$ and $\mathfrak y=\langle C,D,E'\rangle$ are two cardinal characteristics of the continuum, then a \textbf{morphism} $\varphi:\mathfrak x\longrightarrow\mathfrak y$ is just a pair of functions, $\varphi=\langle\varphi_1,\varphi_2\rangle$, satisfying that $\varphi_1:C\longrightarrow A$, $\varphi_2:B\longrightarrow D$, and $(\forall c\in C)(\forall b\in B)(\varphi_1(c)\ E\ b\Rightarrow c\ E'\ \varphi_2(b))$.

Now, suppose that we have a set $X$ and a function $f$ with $\dom(f)=X$. If $F:X^n\longrightarrow X$ is a finitary operation on $X$, we will say that $F$ is \textbf{compatible with $f$} if, for all $x_1,\ldots,x_n,y_1,\ldots,y_n\in X$, 
$[(f(x_1)=f(y_1))\wedge\cdots\wedge(f(x_n)=f(y_n))]\Rightarrow f(F(x_1,\ldots,x_n))=f(F(y_1,\ldots,y_n))$ (that is, we are requiring that the equivalence relation on $X$ given by $x\sim y\iff f(x)=f(y)$ be a congruence with respect to $F$).

Suppose we are given a morphism $\varphi:\mathfrak x\longrightarrow\mathfrak y$ as above, and a finitary operation $F:B^n\longrightarrow B$ which is compatible with $\varphi_2$. We can then define a finitary operation $\varphi(F):D^n\longrightarrow D$ by 
\begin{equation*}
\varphi(F)(d_1,\ldots,d_n)=\begin{cases}
\text{some fixed }d\in D,\text{ if at least one }d_i\notin\ran(\varphi_2); \\
\varphi_2(F(d_1,\ldots,d_n)),\text{ if }d_1=\varphi_2(b_1)\wedge\cdots\wedge d_n=\varphi_2(b_n).
\end{cases}
\end{equation*}
The fact that $F$ is assumed to be compatible with $\varphi_2$ is what ensures that $\varphi(F)$ is well-defined (that is, the definition in the second clause does not depend on the choice of the $b_i$).

\begin{proposition}
Let $\mathfrak x,\mathfrak y$ be two cardinal characteristics, and $\varphi:\mathfrak x\longrightarrow\mathfrak y$ a morphism. Suppose that $\mathscr F$ is a family of finitary functions on $B$ that are compatible with $\varphi_2$, and let $\varphi(\mathscr F)=\{\varphi(F)\big|F\in\mathfrak F\}$. If there exists a $\langle\mathfrak y,\varphi(\mathscr F),\kappa\rangle$-pathway, then there exists a $\langle\mathfrak x,\mathscr F,\kappa\rangle$-pathway.
\end{proposition}

\begin{proof}
Let $\langle D_\alpha\big|\alpha<\kappa\rangle$ be the hypothesized $\langle\mathfrak y,\varphi(\mathscr F),\kappa\rangle$-pathway. %Without loss of generality (letting go of an initial segment of the pathway if necessary) we can assume that 
We define the sequence $\langle B_\alpha\big|\alpha<\kappa\rangle$ by $B_\alpha=\varphi_2^{-1}[D_\alpha]$, and proceed to show that such a sequence is in fact a $\langle\mathfrak x,\mathscr F,\kappa\rangle$-pathway. It is clear that $\bigcup_{\alpha<\kappa}B_\alpha=B$ and that the sequence is increasing and continuous. Now given $\alpha<\kappa$, let $c\in C$ be such that $(\forall d\in D_\alpha)\neg(c\ F\ d)$. It is straightforward to verify that $\varphi_1(c)$ satisfies $(\forall b\in B_\alpha)\neg(\varphi_1(c)\ E\ b)$, for if $b\in B_\alpha=\varphi_2^{-1}[D_\alpha]$ is arbitrary, then we must have that $\varphi_2(b)\in D_\alpha$, which implies (by the assumption on $c$) that $\neg(c\ F\ \varphi_2(b))$, which in turn implies (by the definition of morphism) that $\neg(\varphi_1(c)\ E\ b)$, which establishes the claim.

It just remains to prove that each $B_\alpha$ is closed under $\mathscr F$, so let $F\in\mathscr F$ be an $n$-ary operation and let $b_1,\ldots,b_n\in B_\alpha$. This means that $\varphi_2(b_1),\ldots,\varphi_2(b_n)\in D_\alpha$, so since $D_\alpha$ is closed under $\varphi(F)$, we obtain that 
\begin{equation*}
\varphi_2(F(b_1,\ldots,b_n))=\varphi(F)(\varphi(b_1),\ldots,\varphi(b_n))\in D_\alpha,
\end{equation*}
meaning that $F(b_1,\ldots,b_n)\in\varphi_2^{-1}[D_\alpha]=B_\alpha$, and we are done.
\end{proof}

\section{Pathways in forcing extensions}

We will now focus on $\langle\mathfrak d,\mathscr F,\kappa\rangle$-pathways, since these are the ones that work for both P-points and gruff (and possibly other things too).

\begin{definition}
Let $A=\omega^\omega$, and let $F:A^n\longrightarrow A$ be some finitary operation. We say that $F$ is \textbf{definable} if there is a formula $\varphi$ in the language of set theory, with $n+1$ free variables, such that $(\forall f_1,\ldots,f_n,f\in A)(\varphi(f_1,\ldots,f_n,f)\iff(f=F(f_1,\ldots,f_n))$. This is of course a meta-definition, but we can turn it into a valid $\zfc$ definition by saying that $F$ is definable if and only if it results from G\"odel operations by composition.\footnote{But then the problem, when it comes to doing forcing, is if the defining formula $\varphi$ is not absolute between models of $\zfc$. It's possible that I should say here ``Borel'' instead of ``definable''. In a conversation with Andreas Blass, he suggested that ``definable in $\zfc$ and sufficiently absolute'' (e.g. $\Delta_1^1$) is probably what we need for our purposes.}
\end{definition}

We state a theorem that establishes at once the existence of $\langle\mathfrak d,\mathscr F,\kappa\rangle$-pathways in various forcing extensions, with hypothesis reminiscent of those used by Roitman in~\cite{roitman} (in fact, arguably ``there exists a $\covm$-pathway'' is the right way of formulating the assumption that has in the past been rendered as ``there are cofinally many Cohen reals'').

\begin{theorem}
Let $\mathscr F$ be any family of definable finitary operations on $\omega^\omega$, and suppose that the universe of sets is either some FS iteration of ccc forcings of length $\kappa$ that satisfies $\cf(\kappa)>\omega$, or some CS iteration of proper forcings of length $\kappa=\omega_2$. If cofinally many of the forcings that are being iterated add an unbounded real, then there exists a $\langle\mathfrak d,\mathscr F,\kappa\rangle$-pathway.
\end{theorem}

\begin{proof}
By hypothesis, the universe is of the form $V[G]$ where $G$ is a $V$-generic filter for some forcing notion $\mathbb P_\kappa$ which results from iterating the system $\langle\mathbb P_\alpha,\mathring{\mathbb Q_\alpha}\big|\alpha<\kappa\rangle$. Letting $G_\alpha=G\cap\mathbb P_\alpha$, we define $A_{\alpha+1}=\omega^\omega\cap V[G_\alpha]$ for $\alpha$ a successor (if $\alpha=\bigcup\alpha$, then we just let $A_\alpha=\bigcup_{­\xi<\alpha}A_\xi$ to ensure that the sequence is continuous). Clearly (essentially by definition) the sequence $\langle A_\alpha\big|\alpha<\kappa\rangle$ is continuous and increasing, and to prove that $V[G]\vDash(\omega^\omega=\bigcup_{\alpha<\kappa}A_\alpha)$ we use the fact that $\kappa$ is of uncountable cofinality. Since each $V[G_\alpha]\vDash\zfc$, each of our $A_\alpha$ is closed under $\mathscr F$ (for $\alpha$ limit, we need the fact that each element of $\mathscr F$ is finitary). And since cofinally many of the $\mathbb P_\alpha$ add an unbounded real, clearly for each $A_\alpha$ there is some $f\in V[G]$ such that $(\forall g\in A_\alpha)(f\not\leq^* g)$. Thus $\langle A_\alpha\big|\alpha<\kappa\rangle$ is a $\langle\mathfrak d,\mathscr F,\kappa\rangle$-pathway, and we are done.
\end{proof}

\begin{corollary}
If $\mathscr F$ is the family of all definable operations, then the following models satisfy the existence of $\langle\mathfrak d,\mathscr F,\kappa\rangle$-pathways for the following $\kappa$:
\begin{itemize}
\item Cohen's model ($\kappa$ is the number of Cohen reals that were added),
\item Hechler's model (ditto),
\item Solovay-Tennenbaum (the model for Martin's Axiom), here $\kappa=\mathfrak c$,
\item Kunen-Miller (successively forcing Martin's Axiom together with $\mathfrak c=\aleph_\alpha$, for $\alpha<\omega_1$; here $\kappa=\omega_1<\mathfrak c=\aleph_{\omega_1}$ in the extension),
\item Laver's, Mathias's and Miller's model (in all of these three, $\kappa=\mathfrak c=\omega_2$).
\end{itemize}
\end{corollary}

\begin{proof}
Immediate from the previous theorem. In fact, if we were only concerned with a small number of finitary operations, in most of these cases we would obtain the corresponding pathway just from the fact that the corresponding model satisfies $\mathfrak d=\mathfrak c$. But this result is more powerful, since it gives us pathways even for larger amounts of finitary operations, and because it can also provide pathways in models where $\mathfrak d<\mathfrak c$: for example, if the Cohen or Hechler iteration is ``short'' (i.e. shorter than the size of the continuum in the ground model), then we will get the desired pathway even though $\mathfrak d<\mathfrak c$ in these models. The same phenomenon occurs in the case of the Kunen-Miller model.
\end{proof}

\begin{question}
What about Sacks's model?\footnote{In the version of this document from February 22, 2017, I had also a question about Silver's model here. Of course, now we know that the answer to this is negative, by~\cite{david-osvaldo}.} And of course, the most pressing question right now: what about the Random model?
\end{question}

\section{Empty section}

This document formerly had a somewhat detailed explanation of what exactly the gap in Cohen's proof from~\cite{cohen-random} is. Since such a detailed explanation can now be found in the appendix of~\cite{david-osvaldo}, it is no longer necessary here and so I removed it. I replaced it with a few questions whose answer would be interesting:
\begin{enumerate}
\item Does the existence of a $\covm$-pathway (for appropriate finitary operations) imply the existence of a selective ultrafilter? (there is a sketch for an affirmative answer in Section 5, so this is really just a matter of double-checking details).
\item Does the existence of a $\covm$-pathway (for appropriate finitary operations) imply the existence of a stable ordered union ultrafilter? (this question seems to be harder, and I have not yet been able to find an answer).
\item What other important set-theoretic objects, whose existence follows from some specific cardinal characteristic assumptions, can be constructed from a pathway for said cardinal characteristic?
\item For what other cardinal characteristics of the continuum can we prove consistently that there are pathways while having said cardinal characteristic be strictly smaller than $\mathfrak c$?
\end{enumerate}

\end{document}